\documentclass[11pt]{article}
\usepackage[T1]{fontenc}
\usepackage{lmodern,microtype}
\usepackage[margin=1in]{geometry}
\usepackage{amsmath,amssymb,amsthm}
\usepackage{graphicx}
\usepackage[hidelinks]{hyperref}
\hypersetup{pdftitle={Small additive groups with A + xi A = R},
            pdfauthor={Liang Yu}}

\newtheorem{theorem}{Theorem}[section]
\newtheorem{lemma}[theorem]{Lemma}
\newtheorem{corollary}[theorem]{Corollary}
\theoremstyle{remark}
\newtheorem{remark}[theorem]{Remark}

\newcommand{\R}{\mathbb R}
\newcommand{\Q}{\mathbb Q}
\newcommand{\Z}{\mathbb Z}
\newcommand{\N}{\mathbb N}
\newcommand{\dimH}{\dim_{\mathrm H}}
\newcommand{\HH}{\mathcal H}

\newcommand{\res}{\upharpoonright}

\title{Small additive groups with $A+\xi A=\mathbb R$}
\author{Liang Yu\thanks{Yu was partially supported by NSF of China
No. 12025103 and the Natural Science Foundation of Jiangsu Province
No. BK20243060.}\\[1ex]
\small School of Mathematics\\
\small Nanjing University, Jiangsu Province 210093\\
\small P. R. of China\\
\small\href{mailto:yuliang.nju@gmail.com}{\texttt{yuliang.nju@gmail.com}}}

\begin{document}
\maketitle

\begin{abstract}
We prove that, for every irrational $\xi$, any $F_\sigma$ additive
subgroup of $\R$ has an $F_\sigma$ $\Q$-vector space extension
of the same Hausdorff dimension satisfying $A+\xi A=\R$.
Iteration gives simultaneous surjectivity for any prescribed
countable family of irrational multipliers. Such a space exists
in every prescribed dimension $d\in[0,1)$. We also construct
in ZFC a zero-dimensional lightface $\Delta^0_3$ $\Q$-vector
subspace $A$ for which
$A+\xi A=\R$ holds for every irrational $\xi$, answering both
parts of Question~2 of Ye, Yu, and Zhao. This universal group is
a difference of two $F_\sigma$ sets, but is neither $F_\sigma$
nor $G_\delta$. Both constructions use direct inductions on
finite binary strings and admit effective versions. As applications, we obtain
zero-dimensional $F_\sigma$ groups whose Cartesian squares have
Hausdorff dimension one, and relate their dilate sums to
Marstrand's projection theorem.
\end{abstract}

\noindent\textit{2020 Mathematics Subject Classification.}
Primary 28A78; Secondary 03E15, 28A80.

\section{Introduction}

The relation between algebraic structure and Hausdorff dimension
has been studied since the work of Erd\H{o}s and Volkmann
\cite{EV}, who constructed Borel additive subgroups of $\R$ of
every Hausdorff dimension in $[0,1]$. The situation changes
substantially when closure under multiplication is also required.
Edgar and Miller \cite{EM}, and independently Bourgain
\cite{BourgainRing}, proved that an analytic subring of $\R$
either has Hausdorff dimension zero or equals $\R$. Thus additive
subgroups admit a range of dimensions that is impossible for
analytic subrings. In this paper we study a different way in which
addition and multiplication interact: how small can an additive
subgroup be if the sum of the group and one of its dilates is the
whole real line?

For a subgroup $A$ of $(\R,+)$ and $\xi\in\R$, put
\[
 A+\xi A=\{a+\xi b:a,b\in A\},\qquad
 X_A=\{\xi\in\R:A+\xi A=\R\}.
\]
These sums are also projections: if
$P_\xi(u,v)=u+\xi v$, then $P_\xi(A\times A)=A+\xi A$.
This connects the question with Bourgain's discretized
sum-product and projection theorems \cite{BourgainProjection}.
For example, a consequence of his projection theorem is that,
for a Borel set $E\subseteq\R$ with $0<\dimH E<1$,
\[
 \dimH\{\xi\in\R:\dimH(E+\xi E)\leq\dimH E\}=0.
\]
Indeed, apply \cite[Theorem~4]{BourgainProjection} to
$E\times E$, whose dimension lies between $2\dimH E$ and
$\dimH E+1$. In particular, dimension growth occurs outside a
zero-dimensional set of multipliers. Our question concerns
\emph{exact surjectivity}, rather than dimension growth: even
the equality $\dimH(A+\xi A)=1$ does not by itself imply
$A+\xi A=\R$.

Ye, Yu, and Zhao \cite{YYZ} investigated this surjectivity
problem for additive groups, rings, and fields. Their
Theorem~1.4 gives an $F_\sigma$ additive group of Hausdorff
dimension at most $\frac{1}{2}$ with $X_A\neq\varnothing$.
Their Theorem~1.5 gives, under the continuum hypothesis, a group
of Hausdorff dimension zero satisfying $X_A=\R\setminus\Q$.
In Question~2 they ask whether the continuum hypothesis can be
removed and whether such a group can be Borel. We strengthen
their first construction and answer both questions.

Our first result allows both the dimension and the irrational
multiplier to be prescribed.

\begin{theorem}\label{thm:main}
For every $\xi\in\R\setminus\Q$ and every $d\in[0,1)$, there is an
$F_\sigma$ $\Q$-vector subspace $A\subseteq\R$ such that
\[
 \dimH A=d\qquad\text{and}\qquad A+\xi A=\R.
\]
\end{theorem}

In fact, we prove the following dimension-preserving extension
result: for every $F_\sigma$ subgroup $B\subseteq\R$ and every
irrational $\xi$, there is an $F_\sigma$ subgroup $A$ such that
\[
 B\subseteq A,\qquad \dimH A=\dimH B,\qquad A+\xi A=\R.
\]
This is Theorem~\ref{thm:extension}. Together with the
Erd\H{o}s--Volkmann examples and passage to the divisible hull,
it yields Theorem~\ref{thm:main}. In particular, dimension zero
is already possible for an $F_\sigma$ $\Q$-vector space when
the multiplier is fixed. Iterating the extension theorem gives
simultaneous surjectivity for any prescribed countable set of
irrational multipliers, while preserving the dimension and the
$F_\sigma$ property; see Corollary~\ref{cor:countable-multipliers}.

Our second result uses a single group for every irrational
multiplier.

\begin{theorem}\label{thm:borel-universal}
In ZFC, there is a Borel $\Q$-vector subspace $A\subseteq\R$ such that
\[
 \dimH A=0\qquad\text{and}\qquad X_A=\R\setminus\Q.
\]
Moreover, $A$ can be chosen lightface $\Delta^0_3$, but it is
neither $F_\sigma$ nor $G_\delta$.
\end{theorem}

Theorem~\ref{thm:borel-universal} answers both parts of Question~2
of \cite{YYZ} in ZFC. More precisely, the construction gives $A$
as a difference of two $F_\sigma$ sets. This description and the
effective complexity assertion are proved in
Remarks~\ref{rem:borel-complexity} and
\ref{rem:effective-universal}. It cannot be $F_\sigma$:
\cite[Proposition~4.6]{YYZ}, applied to null $F_\sigma$
subgroups, implies that their sets $X_A$ are meager. It cannot
be $G_\delta$ either, since a dense $G_\delta$ subgroup of
$\R$ must equal $\R$.

The irrationality restriction is necessary for every proper
subgroup. Indeed, if $\xi=p/q\in\Q$, where $p,q\in\Z$ and
$q\neq0$, then $A+\xi A=\R$ would imply
\[
 \R=q\R=qA+pA\subseteq A.
\]
Thus Theorem~\ref{thm:borel-universal} realizes the largest
possible set of surjective multipliers for a proper subgroup.

The examples also clarify the role of Cartesian products in
these questions. The zero-dimensional groups constructed here
satisfy
\[
 \dimH(A\times A)=1>2\dimH A.
\]
This contrasts with the stronger conclusion of Edgar and Miller
\cite[Theorem~1]{EM}: every finite Cartesian power of a proper
analytic subring of $\R$ has Hausdorff dimension zero.
Multiplicative closure is therefore essential to that conclusion.
For our groups, Marstrand's projection theorem \cite{Marstrand}
gives $\dimH(A+\xi A)=1$ for almost every $\xi$, but does not
give surjectivity at this critical product dimension. Our
constructions impose the stronger equality with $\R$ at a
prescribed irrational multiplier, or simultaneously at all
irrational multipliers. These are existence and extension
results, not assertions that every small Borel subgroup has
a surjective dilate sum.

Both proofs use direct inductions on finite binary strings.
For a fixed multiplier, we choose generators $g_\beta$ and
$h_\beta$ satisfying
\[
 h_\beta=\beta-\xi g_\beta,\qquad
 h_\beta+\xi g_\beta=\beta\quad(\beta\in2^\omega).
\]
Surjectivity is built into this identity. Small adjustments of
the assigned generator intervals cancel the variation in
individual linear expressions, allowing all finite sums to be
controlled without increasing the dimension. For the universal
construction, the multiplier is an additional parameter, and
the induction controls finite rational combinations
simultaneously. Both constructions admit effective versions.

Section~\ref{sec:extension} proves the dimension-preserving
extension theorem and discusses its effectivity.
Section~\ref{sec:borel-universal} constructs the universal Borel
group and determines its Borel complexity.
Section~\ref{sec:applications} gives applications to Cartesian
products and relates the examples to Marstrand's projection theorem.

\medskip
\noindent\textit{Notation.}
Following the binary notation of \cite[Section~4]{YYZ}, we regard
elements of $2^\omega$ as reals and finite binary strings as dyadic
rationals. More precisely, in an arithmetic expression we write
\[
 \beta=\sum_{i<\omega}\beta(i)2^{-i-1},\qquad
 \sigma=\sum_{i<|\sigma|}\sigma(i)2^{-i-1}
 \quad(\beta\in2^\omega,\ \sigma\in2^{<\omega}).
\]
Write $\sigma\preceq \beta$ when $\sigma$ is an initial segment of $\beta$,
and let $\beta\res n$ be its initial segment of length $n$. Put
\[
 [\sigma]=\{\beta\in2^\omega:\sigma\preceq \beta\},\qquad
 D_\sigma=[\sigma,\sigma+2^{-|\sigma|}].
\]
Thus $[\sigma]$ is a clopen set of branches, whereas $D_\sigma$ is
the closed interval of their real values. In particular,
\begin{equation}\label{eq:tail}
 \sigma\preceq \beta\quad\Longrightarrow\quad
 0\leq \beta-\sigma\leq2^{-|\sigma|}.
\end{equation}
We keep both expansions of a dyadic real as distinct branch labels.
Compactness and continuity of branch-indexed families refer to the
Cantor topology, not to an asserted identification of topological
spaces. All arithmetic and Hausdorff dimensions are in Euclidean
space. No separate Cantor subset of $\R$ or coding map is needed.

In both proofs we assign a finite binary string to each of finitely
many parameter prefixes. Refining the parameter prefixes allows us
to extend the assigned strings. Extending an assigned string shrinks
its interval $D_\sigma$, so every earlier condition remains true.
The required generators are the unions of the assigned strings
along infinite branches.

We use $\alpha$ (or $\alpha_\sigma,\alpha_i$) for arbitrary test
branches in the finite-stage requirements, and $\beta$ (or
$\beta_\sigma,\beta_i$) for branch variables indexing the final
generators. Both range over $2^\omega$: the change of letter marks
a fresh quantifier, not a passage from a finite approximation to
an infinite branch. In particular, the induction chooses lengths
and assigned finite strings, not the test branches themselves.
The symbols $u$ denote allowed real values in assigned intervals;
$\widetilde u$ denotes a cancellation value. The symbols
$g_\beta,h_\beta$ and $g_{\xi,\beta},h_{\xi,\beta}$ always refer to final
limit generators, never to finite-stage choices.

\section{A dimension-preserving extension}
\label{sec:extension}

\begin{theorem}\label{thm:extension}
Let $B$ be an $F_\sigma$ subgroup of $(\R,+)$, and let
$\xi\in\R\setminus\Q$. There is an $F_\sigma$ subgroup $A$ such that
\[
 B\subseteq A,\qquad \dimH A=\dimH B,\qquad A+\xi A=\R.
\]
\end{theorem}

\medskip
\noindent\textit{Idea of the proof.}
We build surjectivity into the choice of generators. For each
$\beta\in2^\omega$, choose a real $g_\beta$ and put $h_\beta=\beta-\xi g_\beta$.
If $A$ is the group generated by $B$ and all these reals, then
$\beta=h_\beta+\xi g_\beta\in A+\xi A$. Hence $A+\xi A$ contains $[0,1]$
and therefore equals $\R$. The difficulty is to choose the
generators so that adjoining \emph{all their finite integer sums}
does not increase the Hausdorff dimension of $B$.

The key is a small adjustment that makes one linear expression
constant. For a nonzero integer pair $(a,b)$,
\[
 ag_\beta+bh_\beta=(a-\xi b)g_\beta+b\beta,
\]
and $a-\xi b\neq0$ because $\xi$ is irrational. At a finite
stage, let $t$ be the midpoint of a currently assigned interval
for generator values, and let $[v]$ be a sufficiently small
cylinder of parameter branches whose generators are restricted
to that interval. These are different objects: $[v]$ restricts
the parameter, whereas the assigned dyadic interval restricts
the generator value. For an arbitrary test branch $\alpha\in[v]$,
use the cancellation value
\[
 \widetilde u=t+\frac{b(v-\alpha)}{a-\xi b}.
\]
Since $|v-\alpha|\leq2^{-|v|}$, the adjustment of $t$ can be made
arbitrarily small by lengthening $v$. Yet it exactly cancels the
variation in $\alpha$:
\[
 (a-\xi b)\widetilde u+b\alpha=(a-\xi b)t+bv.
\]
The right-hand side is fixed on $[v]$. Doing this on finitely many
cylinders reduces the sums for one coefficient choice to a finite
set $S$.

Write $d=\dimH B$ and exhaust $B$ by increasing compact sets $B_j$.
The set $B_j+S$ is a finite union of translates of $B_j$, so it
admits an open cover with arbitrarily small
$\bigl(d+\frac{1}{j}\bigr)$-content. We then choose sufficiently
small dyadic intervals near the cancellation values so that
\emph{every} choice from these intervals keeps the sums inside
that cover. Later steps only shrink the assigned intervals.
Thus the exact cancellation need not persist, but the cover does.
This lets us treat the finitely many coefficient choices at each
stage one after another without losing the conclusions already
obtained.

Finally, longer parameter prefixes separate any fixed finite set
of distinct branches, while increasing coefficient bounds eventually
include every finite integer combination. The nested assigned
strings define continuous families of generators. Their finite
sums, together with the compact sets $B_j$, give an increasing
compact exhaustion of $A$. Each fixed compact family is covered
at all sufficiently late stages by covers whose diameters and
contents tend to zero. Consequently $A$ is $F_\sigma$ and
$\dimH A\leq d$; the inclusion $B\subseteq A$ gives equality.

\medskip
Figure~\ref{fig:parameter-relationships} gives a guide to the
notation used in the proof below. It separates prefixes of the
parameter branch from the assigned binary strings that restrict
the generator values.

\begin{figure}[p]
\centering
\includegraphics[width=\textwidth]{ 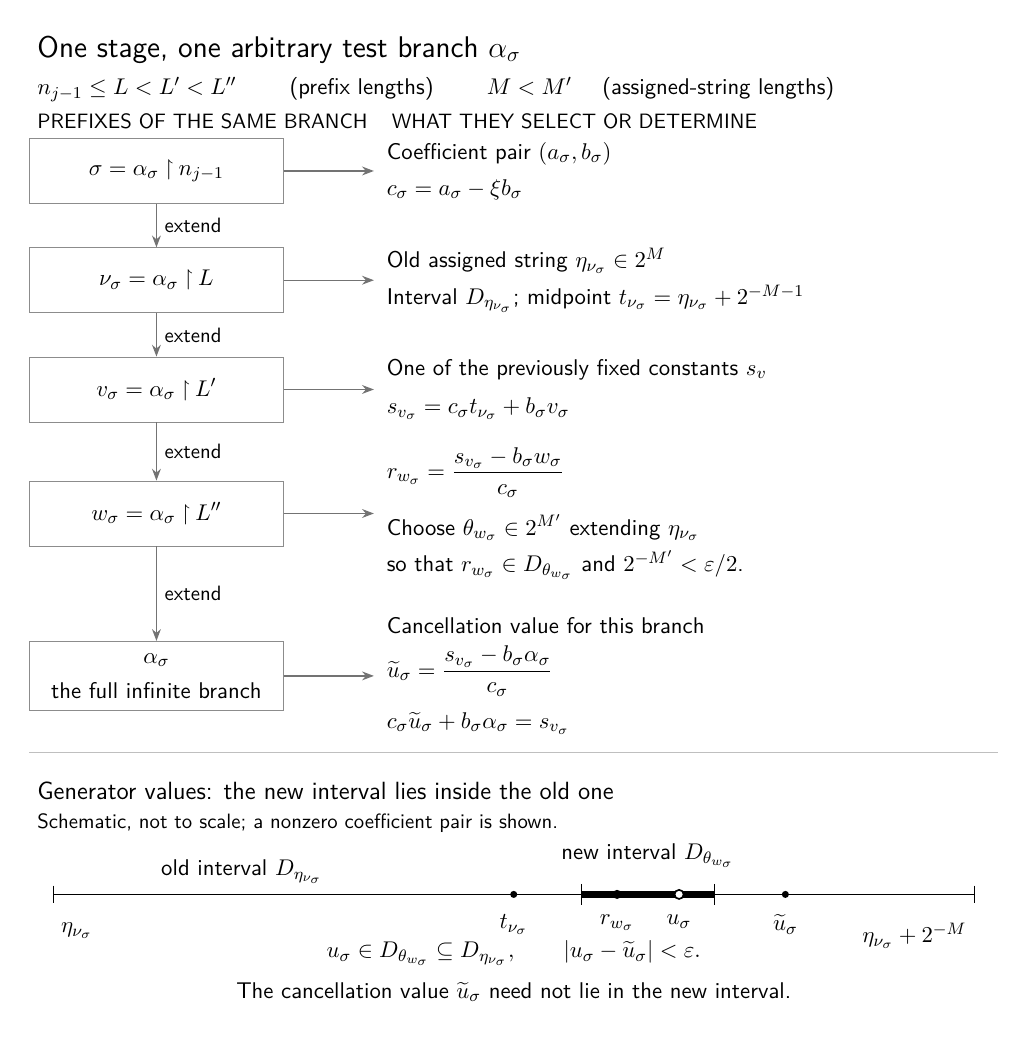}
\caption{Parameters for a fixed nonzero coefficient pair during
stage $j$. Read downward for longer prefixes of the same arbitrary test branch
$\alpha_\sigma$, and rightward for the objects they select. The assigned
string $\theta_{w_\sigma}$ extends $\eta_{\nu_\sigma}$, not
$w_\sigma$. The interval sketch distinguishes an arbitrary allowed
value $u_\sigma$ from the cancellation value
$\widetilde u_\sigma$, which need not lie in the new interval.}
\label{fig:parameter-relationships}
\end{figure}

\medskip
\begin{proof}
Put $d=\dimH B$, and choose increasing nonempty compact sets $B_j$
such that $0\in B_1$ and $B=\bigcup_{j\geq1}B_j$.
For each $\beta\in2^\omega$ we shall construct a real $g_\beta$ and put
\[
 h_\beta=\beta-\xi g_\beta.
\]
Thus $h_\beta+\xi g_\beta=\beta$ automatically. We have only to keep the group
generated by $B$ and these reals small.

At stage $j$ choose integers $n_j,m_j$, and a string
\[
 \tau^j_\sigma\in2^{m_j}\qquad(\sigma\in2^{n_j}).
\]
We require $n_j>n_{j-1}$, $m_j>m_{j-1}$, and
\begin{equation}\label{eq:string-nesting}
 \tau^{j-1}_{\sigma\res n_{j-1}}\preceq\tau^j_\sigma.
\end{equation}
Start with $n_0=0$, $m_0=2$, and $\tau^0_\varnothing=01$.
Thus all the intervals for $g_\beta$ lie in $D_{01}=[1/4,1/2]$.
At the end we set
\begin{equation}\label{eq:branch-generators}
 g_\beta=\bigcup_{\ell\geq0}\tau^\ell_{\beta\res n_\ell},
 \qquad h_\beta=\beta-\xi g_\beta.
\end{equation}
Here $\beta$ ranges over all of $2^\omega$, independently of any
finite stage. The union is an infinite binary string, also read
as a real; the index $\ell$ ranges over the entire construction.

\medskip
\noindent\textit{Stage $j$.}
Keep the old prefixes $\sigma\in2^{n_{j-1}}$ fixed throughout this
stage. Consider every choice of integer pairs
\[
 (a_\sigma,b_\sigma),\qquad |a_\sigma|,|b_\sigma|\leq j.
\]
There are exactly
\[
 q_j=(2j+1)^{2\cdot2^{n_{j-1}}}
\]
coefficient lists. Stage $j$ consists of exactly $q_j$ substages,
one for each list, including the all-zero list. Fix any ordering
of these lists and process them successively.
We shall obtain a finite family $\mathcal I_j$ of open intervals with
\begin{equation}\label{eq:stage-cover}
 |I|<2^{-j}\quad(I\in\mathcal I_j),\qquad
 \sum_{I\in\mathcal I_j}|I|^{d+\frac{1}{j}}<2^{-j},
\end{equation}
covering every number
\begin{equation}\label{eq:allowed-sums}
 \begin{gathered}
 z+\sum_{\sigma\in2^{n_{j-1}}}
       \bigl((a_\sigma-\xi b_\sigma)u_\sigma
                  +b_\sigma\alpha_\sigma\bigr),\\
 z\in B_j,\qquad \alpha_\sigma\in[\sigma],\qquad
 u_\sigma\in D_{\tau^j_{\alpha_\sigma\res n_j}}.
 \end{gathered}
\end{equation}
for all the coefficient choices above. Each $\alpha_\sigma$ is
an arbitrary infinite extension of $\sigma$, not a branch chosen
by the induction: the requirement is universal over all such
test branches and all allowed $u_\sigma$. Zero pairs leave prefixes
unused. Since the assertion holds for every point of each assigned
interval, it survives every later extension of the strings.

At each of the $q_j$ substages, we treat its coefficient list
and allocate content less than $2^{-j}/q_j$ to its cover.
The current data at a substage are the output of the preceding
substage (or of stage $j-1$ at the first substage). Write
$L$ for the current parameter-prefix length and denote the current
assigned strings by $\eta_\nu\in2^M$ for $\nu\in2^L$.
Initially,
\[
 L=n_{j-1},\qquad M=m_{j-1},\qquad
 \eta_\nu=\tau^{j-1}_\nu\quad(\nu\in2^{n_{j-1}}).
\]
After each substage, update $L,M$ and the assigned
strings to the refinements just obtained. The old length
$n_{j-1}$ and the coefficient indices $\sigma\in2^{n_{j-1}}$
remain fixed throughout the stage. Thus $\sigma$ determines the
coefficient pair, while a current prefix $\nu\in2^L$ determines
the current assigned interval; they satisfy
$\sigma=\nu\res n_{j-1}$. We reserve $n_j,m_j,\tau^j$ for the
final data at the end of the stage.

At the substage for the all-zero list, simply cover $B_j$ with
the stated bounds; no refinement is needed. Now consider a
substage with a nonzero coefficient list, and
put \[c_\sigma=a_\sigma-\xi b_\sigma.\] By the irrationality of $\xi$,
$c_\sigma\neq0$ whenever $(a_\sigma,b_\sigma)\neq(0,0)$.
For each current prefix $\nu\in2^L$, let
\[
 t_\nu=\eta_\nu+2^{-M-1},\qquad \rho=2^{-M-1};
 \quad D_{\eta_\nu}=[\eta_\nu, \eta_\nu+2^{-M}]=[t_\nu-\rho,t_\nu+\rho].
\]

We first refine the parameter prefixes so that the forthcoming
cancellation requires only a small adjustment of the old midpoint.
By using a longer prefix $v$ instead, we reduce the approximation
error to at most $2^{-L'}$ and can make the correction smaller
than $\rho/4$. At this point we refine only the parameter
prefixes, leaving the old assigned intervals $D_{\eta_\nu}$
unchanged.

Choose $L'>L$ and, for each $\nu\in2^L$, consider all its
length-$L'$ extensions $v$, thereby subdividing $[\nu]$ into
the finer cylinders $[v]$.
For every resulting prefix $v\in2^{L'}$, write
$\nu=v\res L$ and $\sigma=v\res n_{j-1}$.
Choose $L'$ large enough that
\[
 |b_\sigma|2^{-L'}<|c_\sigma|\rho/4
 \quad\text{whenever }(a_\sigma,b_\sigma)\neq(0,0).
\]
For every $v\in2^{L'}$, put
\[
 s_v=c_{v\res n_{j-1}}t_{v\res L}
       +b_{v\res n_{j-1}}v
     =c_\sigma t_\nu+b_\sigma v.
\]
Here $v$ in the arithmetic expression denotes its dyadic rational
value. This definition also applies when the coefficient pair is
zero, in which case $s_v=0$.
For a nonzero pair and each $\alpha\in[v]$, define the cancellation value
\begin{equation}\label{eq:cancellation-value}
 \widetilde u=\frac{s_v-b_\sigma \alpha}{c_\sigma}=t_{\nu}+\frac{b_{\sigma}}{c_{\sigma}}(v-\alpha).
\end{equation}
It is called the cancellation value because the term $-b_\sigma \alpha$
in $c_\sigma\widetilde u=s_v-b_\sigma \alpha$ cancels the added
$b_\sigma \alpha$, leaving the constant $s_v$.
Here $v$, and hence $\sigma$, is fixed, and $\widetilde u$ depends
on $v$ and the chosen $\alpha$; it is not one common value for all
prefixes. The tilde distinguishes this cancellation value
from the arbitrary allowed values $u_\sigma$ in
\eqref{eq:allowed-sums}. We have
\[
 |\widetilde u-t_\nu|
 =\frac{|b_\sigma|}{|c_\sigma|}|v-\alpha|
 \leq\frac{|b_\sigma|}{|c_\sigma|}2^{-L'}<\rho/4, \mbox{ and }\]
 
 \[
 \qquad c_\sigma\widetilde u+b_\sigma \alpha=s_v.
\]
Thus $\widetilde u$ stays close to the old midpoint while its contribution
$c_\sigma\widetilde u+b_\sigma \alpha$ has the same value $s_v$ for every
$\alpha\in[v]$. There are only finitely many prefixes $v$, so the
possible contributions have been reduced to finitely many numbers.
\emph{This is the key idea of the proof.}
We do not require the final generator on this branch to equal
this cancellation value exactly; below we choose its allowed interval sufficiently
close to this value to keep the resulting sums inside the chosen
open cover. For a zero pair the contribution is already the constant
$s_v=0$, and no division is made.

With the coefficient list, the current assignment, and $L'$ fixed,
define
\[
 S=\left\{
   \sum_{\sigma\in2^{n_{j-1}}}s_{v_\sigma}
   \;\middle|\;
   \begin{array}{l}
   v_\sigma\in2^{L'},\quad v_\sigma\res n_{j-1}=\sigma\\
   \text{for every }\sigma\in2^{n_{j-1}}
   \end{array}
   \right\}.
\]
The tuple $(v_\sigma)_{\sigma\in2^{n_{j-1}}}$ ranges over all
independent choices of one length-$L'$ extension of each old
prefix $\sigma$. In particular, these strings are not fixed
prefixes of previously chosen branches. For each such choice,
the summand is the already defined number
\[
 s_{v_\sigma}=c_\sigma t_{v_\sigma\res L}+b_\sigma v_\sigma.
\]
Thus forming a sum in $S$ means selecting one constant above each
old prefix and adding the selected constants. No infinite branch,
allowed value $u_\sigma$, or element of $B_j$ enters the definition
of $S$; the latter is added separately in $B_j+S$. The set $S$
is defined anew for each coefficient choice from its current data.

For each $\sigma\in2^{n_{j-1}}$, there are exactly
$2^{L'-n_{j-1}}$ strings of length $L'$ extending $\sigma$.
A sum in $S$ is obtained by choosing one such extension for each
of the $2^{n_{j-1}}$ old prefixes. Consequently,
\[
 1\leq |S|\leq
 \left(2^{L'-n_{j-1}}\right)^{2^{n_{j-1}}}
 =2^{(L'-n_{j-1})\,2^{n_{j-1}}}.
\]
The inequality need not be an equality, since different tuples
of extensions may give the same sum.
 
This estimate also explains how to choose the required cover of
$B_j+S$. Since $\dim_{\mathrm H}B_j\leq d$, we have
$\mathcal H^{d+\frac{1}{j}}(B_j)=0$. By compactness, choose a finite
family $\mathcal U$ of open intervals covering $B_j$ such that
\[
 |I|<2^{-j}\quad(I\in\mathcal U),\qquad
 \sum_{I\in\mathcal U}|I|^{d+\frac{1}{j}}
 <\frac{2^{-j}}{q_j|S|}.
\]
Translate this cover by each element of $S$, and put
\[
 \mathcal I=\{I+s:I\in\mathcal U,\ s\in S\}.
\]
Then $\mathcal I$ is a finite open interval cover of $B_j+S$.
Translation preserves interval lengths, so
\[
 |J|<2^{-j}\quad(J\in\mathcal I),\qquad
 \sum_{J\in\mathcal I}|J|^{d+\frac{1}{j}}
 \leq |S|\sum_{I\in\mathcal U}|I|^{d+\frac{1}{j}}
 <\frac{2^{-j}}{q_j}.
\]
Thus the number of translates may be large, but it is finite,
and we compensate by choosing the initial cover sufficiently small.

Write $O=\bigcup\mathcal I$ and
\[
 \delta=\operatorname{dist}(B_j+S,\R\setminus O)>0,\qquad
 C=\sum_\sigma|c_\sigma|>0.
\]
Choose
\[
 0<\varepsilon<\min\{\rho/4,\delta/(2C)\}.
\]

We now use the cancellation values to choose finite strings.
Choose $L''>L'$ large enough that
\[
 \frac{|b_\sigma|}{|c_\sigma|}2^{-L''}<\varepsilon/2
\]
for every nonzero pair. For each $w\in2^{L''}$, let
$v=w\res L'$, $\nu=w\res L$, and $\sigma=w\res n_{j-1}$.
If its coefficient pair is nonzero, put
\[
 r_w=\frac{s_v-b_\sigma w}{c_\sigma}=t_{\nu}+\frac{b_{\sigma}}{c_{\sigma}}(v-w).
\]
Since $s_v=c_\sigma t_\nu+b_\sigma v$ and $w$ extends $v$,
\[
 |r_w-t_\nu|
 =\frac{|b_\sigma|}{|c_\sigma|}|v-w|
 \leq\frac{|b_\sigma|}{|c_\sigma|}2^{-L'}
 <\frac{\rho}{4}.
\]
Thus $r_w$ lies in the interior of
$D_{\eta_\nu}=[t_\nu-\rho,t_\nu+\rho]$.
By the definition of $D_{\eta_\nu}$, every point of this interval
is the real value of a binary sequence extending $\eta_\nu$.
Choose such a sequence $\gamma_w\in[\eta_\nu]$ whose real value
is $r_w$.

Choose one integer $M'>M$ such that
\[
 2^{-M'}<\frac{\varepsilon}{2},
\]
and, for every $w$ with nonzero coefficient pair, set
$\theta_w=\gamma_w\res M'$.
Then $\theta_w$ extends $\eta_\nu$, and
\[
 r_w\in D_{\theta_w}=[\theta_w, \theta_w+2^{-M'}],\qquad
 \operatorname{diam}(D_{\theta_w})=2^{-M'}.
\]
Consequently, every $u\in D_{\theta_w}$ satisfies
\[
 \left|u-\frac{s_v-b_\sigma w}{c_\sigma}\right|
 =|u-r_w|\leq2^{-M'}<\frac{\varepsilon}{2}.
\]
For a zero coefficient pair, let $\theta_w$ be $\eta_\nu$
followed by $M'-M$ zeros. All the strings $\theta_w$ therefore
have the same length $M'$ and extend their previously assigned
strings. A dyadic value of $r_w$ causes no difficulty: it may be
an endpoint of $D_{\theta_w}$, since the estimate uses only the
diameter of that interval.

By \eqref{eq:tail}, whenever $\alpha\in[w]$ and $u\in D_{\theta_w}$,
\[
 \begin{aligned}
 \left|u-\frac{s_v-b_\sigma\alpha}{c_\sigma}\right|
 &\leq |u-r_w|
       +\left|r_w-\frac{s_v-b_\sigma\alpha}{c_\sigma}\right|\\
 &<\frac{\varepsilon}{2}+\frac{\varepsilon}{2}
   =\varepsilon.
 \end{aligned}
\]
 and so 
\[ |c_\sigma u+b_\sigma \alpha-s_v|<|c_\sigma|\varepsilon
\]
for a nonzero pair; the contribution of a zero pair is zero.

To check explicitly that this gives the required cover, take a sum
of the form \eqref{eq:allowed-sums}, with the new intervals:
\[
 T=z+\sum_{\sigma\in2^{n_{j-1}}}
       (c_\sigma u_\sigma+b_\sigma \alpha_\sigma),
 \qquad z\in B_j,\quad \alpha_\sigma\in[\sigma].
\]
For each $\sigma$, put
\[
 \nu_\sigma=\alpha_\sigma\res L,\qquad
 v_\sigma=\alpha_\sigma\res L',\qquad
 w_\sigma=\alpha_\sigma\res L''.
\]
Since $\alpha_\sigma\in[\sigma]$, these prefixes satisfy
\[
 v_\sigma\res n_{j-1}=\sigma,
 \qquad v_\sigma\res L=\nu_\sigma.
\]
We can therefore apply the earlier definition of $s_v$ with
$v=v_\sigma$ and $\nu=\nu_\sigma$, obtaining
\[
 s_{v_\sigma}=c_\sigma t_{\nu_\sigma}+b_\sigma v_\sigma,
 \qquad
 t_{\nu_\sigma}=\eta_{\nu_\sigma}+2^{-M-1}.
\]
In this formula, $t_{\nu_\sigma}$ is the midpoint of the old
assigned interval $D_{\eta_{\nu_\sigma}}$, and $v_\sigma$ is read
as its dyadic rational value. Thus $s_{v_\sigma}$ is not a new
choice: it is the constant already attached to the length-$L'$
prefix of $\alpha_\sigma$. The chosen prefixes $v_\sigma$ select one
particular sum from the finite set $S$.

The allowed value $u_\sigma$ belongs to $D_{\theta_{w_\sigma}}$.
For each nonzero coefficient pair, apply
\eqref{eq:cancellation-value} with $v=v_\sigma$ and $\alpha=\alpha_\sigma$,
and denote the resulting cancellation value by $\widetilde u_\sigma$:
\[
 \widetilde u_\sigma
 =\frac{s_{v_\sigma}-b_\sigma \alpha_\sigma}{c_\sigma}
 =t_{\nu_\sigma}
   +\frac{b_\sigma}{c_\sigma}(v_\sigma-\alpha_\sigma).
\]
Then the preceding estimate and the defining identity give
\[
 |u_\sigma-\widetilde u_\sigma|<\varepsilon,
 \qquad
 c_\sigma\widetilde u_\sigma+b_\sigma \alpha_\sigma=s_{v_\sigma}.
\]
The allowed value $u_\sigma$ need not equal $\widetilde u_\sigma$;
it is only required to approximate it. Consequently,
\[
 \left|c_\sigma u_\sigma+b_\sigma \alpha_\sigma-s_{v_\sigma}\right|
 =|c_\sigma|\,|u_\sigma-\widetilde u_\sigma|
 \leq |c_\sigma|\varepsilon.
\]
The same bound holds for a zero coefficient pair, since then
$c_\sigma=b_\sigma=s_{v_\sigma}=0$; in that case no cancellation
value needs to be defined.

Compare $T$ with the particular point
\[
 T_0=z+\sum_{\sigma\in2^{n_{j-1}}}s_{v_\sigma}.
\]
Each $v_\sigma$ is a length-$L'$ extension of $\sigma$, so the
definition of $S$ gives $\sum_\sigma s_{v_\sigma}\in S$.
As $z\in B_j$, it follows that $T_0\in B_j+S$.
The same $z$ occurs in both sums and cancels upon subtraction.
Thus, with all sums below taken over $\sigma\in2^{n_{j-1}}$,
\begin{align*}
 |T-T_0|
 &=\left|\sum_\sigma
       (c_\sigma u_\sigma+b_\sigma \alpha_\sigma-s_{v_\sigma})\right|\\
 &\leq\sum_\sigma
       \left|c_\sigma u_\sigma+b_\sigma \alpha_\sigma-s_{v_\sigma}\right|\\
 &\leq\varepsilon\sum_\sigma|c_\sigma|
   =C\varepsilon<\frac{\delta}{2}.
\end{align*}

Finally, $\delta=\operatorname{dist}(B_j+S,\R\setminus O)>0$
is a uniform distance from the compact set $B_j+S$ to the
complement of its open cover $O$. If $T\notin O$, then
$T\in\R\setminus O$, and $T_0\in B_j+S$ would imply
$|T-T_0|\geq\delta$, contradicting the estimate above.
Hence $T\in O$. This proves the required covering assertion for
every allowed sum, not just for the cancellation values.
Since each new assigned interval lies inside its
predecessor, all covering assertions obtained earlier are
preserved as well.

Now replace $L$ by $L''$, replace $M$ by the common length of the
strings $\theta_w$, and use $\eta_w=\theta_w$ as the current
assignment for the next substage.
After all $q_j=(2j+1)^{2\cdot2^{n_{j-1}}}$ substages have been
completed---that is, after every one of the $q_j$ coefficient
lists has been treated---set
\[
 n_j=L,\qquad m_j=M,\qquad
 \tau^j_\nu=\eta_\nu\quad(\nu\in2^{n_j}).
\]
Both lengths have increased, and every assigned string extends
the corresponding stage-$(j-1)$ string.
The union of the $q_j$ substage covers gives \eqref{eq:stage-cover} and
\eqref{eq:allowed-sums}. This completes the induction.

\medskip
\noindent\textit{Verification.}
The induction is now complete. For every $\beta\in2^\omega$,
let $g_\beta,h_\beta$ be the final generators given by
\eqref{eq:branch-generators}, and put
\[
 E=\{g_\beta:\beta\in2^\omega\}
       \cup\{h_\beta:\beta\in2^\omega\},
 \qquad A=B+\langle E\rangle_{\Z}.
\]
Thus $E$ contains the final generators for all infinite branches;
the variable $\beta$ is not a parameter fixed at a particular stage.
If two branches agree through $n_j$, their $g$-values belong to
the same interval of length $2^{-m_j}$. Their numerical values
are within $2^{-n_j}$, so their $h$-values are within
$2^{-n_j}+|\xi|2^{-m_j}$. Thus the two indexed families are continuous
in the Cantor topology, and $E$ is compact.

For each $j\geq1$, form $K_j$ from these same fixed generators,
using fresh branch variables $\beta_\sigma$:
\[
K_j=
\left\{
z+\sum_{\sigma\in2^{n_{j-1}}}
   \bigl(a_\sigma g_{\beta_\sigma}
         +b_\sigma h_{\beta_\sigma}\bigr)
\;\middle|\;
\begin{array}{l}
z\in B_j,\\[2pt]
a_\sigma,b_\sigma\in\mathbb Z,\quad
|a_\sigma|,|b_\sigma|\leq j,\\[2pt]
\beta_\sigma\in2^\omega,\quad
\beta_\sigma\upharpoonright n_{j-1}=\sigma\\
\text{for every }\sigma\in2^{n_{j-1}}
\end{array}
\right\}.
\]
Here each $\beta_\sigma$ ranges independently over all infinite
extensions of $\sigma$; these branches are not inherited from
choices made during the induction. The index $j$ specifies
$B_j$, the prefix length, and the coefficient bounds, whereas
$g_{\beta_\sigma}$ and $h_{\beta_\sigma}$ are always the final
limit generators. Thus $K_j\subseteq A$, and $K_j$ is not the
set obtained by allowing generator values to range freely over
the stage-$j$ dyadic intervals.

For each fixed coefficient list, the set of sums is the continuous
image of the compact space
\[
 B_j\times\prod_{\sigma\in2^{n_{j-1}}}[\sigma].
\]
There are finitely many coefficient lists, so $K_j$ is compact.
To see that the stage-$j$ cover covers $K_j$, take any branches
$\beta_\sigma$ occurring in its definition and apply
\eqref{eq:allowed-sums} with
\[
 \alpha_\sigma=\beta_\sigma,\qquad
 u_\sigma=g_{\beta_\sigma}
       \in D_{\tau^j_{\beta_\sigma\res n_j}}.
\]
The interval membership follows from the nesting of the assigned
strings, and
\[
 a_\sigma g_{\beta_\sigma}+b_\sigma h_{\beta_\sigma}
 =(a_\sigma-\xi b_\sigma)g_{\beta_\sigma}
       +b_\sigma\beta_\sigma.
\]
Hence every defining sum for $K_j$ is one of the allowed sums.
Also $K_j\subseteq K_{j+1}$: retain each used branch and its
coefficients on its longer prefix, and assign zero pairs elsewhere.
Finally,
\begin{equation}\label{eq:group-union}
 A=\bigcup_{j\geq1}K_j.
\end{equation}
Indeed, in any finite sum, combine terms with the same branch label.
All remaining labels have distinct sufficiently long prefixes,
and all coefficients and the chosen element of $B$ are allowed at
every sufficiently large stage. This also proves that $A$ is
$F_\sigma$.

Fix $r$ and $s>d$. For all sufficiently large $j\geq r$,
$d+\frac{1}{j}<s$, and the cover of $K_j$ also covers $K_r$. Hence
\[
 \sum_{I\in\mathcal I_j}|I|^s
 \leq\sum_{I\in\mathcal I_j}|I|^{d+\frac{1}{j}}<2^{-j},
 \qquad \max_{I\in\mathcal I_j}|I|<2^{-j}.
\]
It follows that $\HH^s(K_r)=0$. By \eqref{eq:group-union},
$\dimH A\leq d$; the inclusion $B\subseteq A$ gives equality.
For every $\beta\in2^\omega$ we have
\[
 \beta=h_\beta+\xi g_\beta\in A+\xi A.
\]
The numerical values of these $\beta$ fill $[0,1]$. Thus $A+\xi A$ is
a subgroup containing an interval, and consequently equals $\R$.
\end{proof}

For the initial group needed in Theorem~\ref{thm:main}, we use
the classical construction of Erd\H{o}s and Volkmann.

\begin{lemma}[Erd\H{o}s--Volkmann, {\cite[Satz~1]{EV}}]
\label{lem:initial-groups}
For every $d\in[0,1)$ there exists an $F_\sigma$ additive subgroup
$B\subseteq\R$ with $\dimH B=d$.
\end{lemma}

For $d=0$, take $B=\{0\}$. For $0<d<1$, the assertion follows
from \cite[Satz~1]{EV}; the $F_\sigma$ property follows from the
digit description of the groups in its proof.

\begin{proof}[Proof of Theorem~\ref{thm:main}]
Choose the initial group $B$ by Lemma~\ref{lem:initial-groups}, and
apply Theorem~\ref{thm:extension} to obtain an $F_\sigma$ subgroup
$C\supseteq B$ with $\dimH C=d$ and $C+\xi C=\R$.
Take its divisible hull
\[
 A=C^{\mathrm{div}}
   :=\bigcup_{n\geq1}\frac{1}{n!}C
   =\operatorname{span}_{\Q}C.
\]
The groups $\frac{1}{n!}C$ are increasing, and every positive
integer divides some factorial. Their union is therefore a
$\Q$-vector space. It is $F_\sigma$, since each scaled copy of
$C$ is $F_\sigma$, and countable stability of Hausdorff dimension
gives
\[
 \dimH A=\sup_{n\geq1}\dimH\left(\frac{1}{n!}C\right)
        =\dimH C=d.
\]
Finally, $C\subseteq A$ implies
$\R=C+\xi C\subseteq A+\xi A\subseteq\R$.
\end{proof}

\begin{corollary}[Countably many prescribed multipliers]
\label{cor:countable-multipliers}
Let $B\subseteq\R$ be an $F_\sigma$ additive subgroup, and let
$D\subseteq\R\setminus\Q$ be countable. There is an $F_\sigma$
$\Q$-vector subspace $A\subseteq\R$ such that
\[
 B\subseteq A,\qquad \dimH A=\dimH B,\qquad D\subseteq X_A.
\]
In particular, for every $d\in[0,1)$ and every countable
$D\subseteq\R\setminus\Q$, such a space can be chosen with
$\dimH A=d$.
\end{corollary}

\begin{proof}
If $D=\varnothing$, take $A=B^{\mathrm{div}}$. Otherwise, write
$D=\{\xi_n:n\geq1\}$, allowing repetitions when $D$ is finite.
Starting with $C_0=B$, apply Theorem~\ref{thm:extension}
successively to obtain $F_\sigma$ subgroups $C_n$ satisfying
\[
 C_{n-1}\subseteq C_n,\qquad
 \dimH C_n=\dimH B,\qquad C_n+\xi_n C_n=\R.
\]
Then $C=\bigcup_{n\geq0}C_n$ is an $F_\sigma$ subgroup of the
same dimension as $B$. For every $n\geq1$,
$C+\xi_n C\supseteq C_n+\xi_n C_n=\R$.
Taking $A=C^{\mathrm{div}}$, as in the preceding proof,
preserves the dimension, the $F_\sigma$ property, and all these
surjectivity assertions. The final statement follows by
choosing $B$ from Lemma~\ref{lem:initial-groups}.
\end{proof}

Thus even a prescribed countable dense set of irrational
multipliers is compatible with an $F_\sigma$ group of dimension
zero. Theorem~\ref{thm:borel-universal} treats all irrational
multipliers at once, but its group cannot be $F_\sigma$.

\begin{remark}[Effectivity]\label{rem:effective-construction}
For the effective descriptive-set-theoretic terminology used
here, see Moschovakis \cite[Chapter~3]{Moschovakis}.
The construction in Theorem~\ref{thm:extension} relativizes.
Suppose that $B$ is a $\Sigma^0_2(x)$ additive subgroup of $\R$,
$d=\dimH B$, and $\xi$ is irrational. Then there is an additive
subgroup $A\in\Sigma^0_2(x\oplus d\oplus\xi)$ such that
\[
 B\subseteq A,\qquad \dimH A=d,\qquad A+\xi A=\R.
\]
Here reals used as oracles are given by standard codes, and
$\Sigma^0_2(z)$ means a union of a uniformly $\Pi^0_1(z)$
sequence of sets. Put $z=x\oplus d\oplus\xi$.

Write $B=\bigcup_{i\geq1}F_i$, where the $F_i$ are uniformly
$\Pi^0_1(x)$, and take
\[
 B_j=[-j,j]\cap\left(\{0\}\cup\bigcup_{i\leq j}F_i\right).
\]
These form an increasing compact exhaustion of $B$ and are
uniformly effectively compact relative to $x$: finite rational
open covers of $B_j$ can be enumerated relative to $x$.
Indeed, such a family covers $B_j$ exactly when, together with
the enumerated complement of $B_j$, it covers $[-j,j]$.

Using the oracle $d$, choose rational exponents $s_j$ with
\[
 d<s_j<d+\frac{1}{j},
\]
and replace $d+\frac{1}{j}$ at stage $j$ by $s_j$.
Every nonzero $c_\sigma=a_\sigma-\xi b_\sigma$ is a nonzero
$z$-computable real, so a positive rational lower bound for
$|c_\sigma|$ can be found by successive approximation. The
required prefix lengths and the constants $s_v$ are therefore
computable relative to $z$. Keep the sums defining $S$ as a
finite list, allowing repetitions, to avoid deciding equality
between computable reals. The compact set $B_j+S$ is effectively
compact relative to $z$, and $\mathcal H^{s_j}(B_j+S)=0$.
Thus a search through finite rational open covers finds a family
$\mathcal I$ with
\[
 |I|<2^{-j}\quad(I\in\mathcal I),\qquad
 \sum_{I\in\mathcal I}|I|^{s_j}<\frac{2^{-j}}{q_j}.
\]
Both the covering condition and the strict content bound are
computably enumerable relative to $z$, and existence guarantees that the
search terminates. A positive rational safety margin can likewise
be found by searching for a common rational amount by which all
the intervals can be shrunk while still covering $B_j+S$.
Use this margin in place of $\delta$ when choosing a rational
$\varepsilon>0$.

One choice in the proof needs a small reformulation: we should
not try to compute a binary expansion of $r_w$, since this cannot
be done uniformly at dyadic rationals. Instead, dovetail a search
over finite extensions $\theta$ of $\eta_\nu$, of length greater
than $M$, for which
\[
 r_w-\frac{\varepsilon}{2}<\theta,
 \qquad
 \theta+2^{-|\theta|}<r_w+\frac{\varepsilon}{2}.
\]
These strict inequalities are computably enumerable relative to $z$, and a suitable
extension exists because $r_w$ lies in the interior of
$D_{\eta_\nu}$. Thus every $u\in D_\theta$ satisfies
$|u-r_w|<\varepsilon/2$, which is all the subsequent argument
uses; we do not require $r_w\in D_\theta$. After making these
finitely many choices, extend all the strings to a common length.
Zero coefficient pairs are handled by appending zeros as before.

Consequently the stage data are computable relative to $z$, and the maps
$\beta\mapsto g_\beta$ and $\beta\mapsto h_\beta$ from Cantor space to $\R$ are
$z$-computable, with effective moduli supplied by $n_j$ and $m_j$.
Hence $K=\{0\}\cup E$, with $E$ the full final generator set
defined above, is effectively compact relative to $z$, and
\[
 A=B+\langle K\rangle_{\Z}
   =\bigcup_{j,m\geq1}\bigl(B_j+m(K-K)\bigr),
\]
Here $K-K=\{u-v:u,v\in K\}$, and for $m\geq1$ we write
\[
 m(K-K)
 =\underbrace{(K-K)+\cdots+(K-K)}_{m\text{ summands}}
 =\left\{\sum_{i=1}^m(u_i-v_i):u_i,v_i\in K\right\}.
\]
Thus $m(K-K)$ denotes an $m$-fold Minkowski sum. Since $0\in K$, repetitions and
zero terms allow every finite integer combination of elements of
$K$ to be written in this form for some $m$, which justifies the
union above. The sets in this union are uniformly effectively
compact relative to $z$, hence uniformly $\Pi^0_1(z)$.
Thus $A\in\Sigma^0_2(z)$.
Since $s_j\to d$, the dimension argument is unchanged, as is
the identity $h_\beta+\xi g_\beta=\beta$ proving surjectivity.
All searches use only $z$.
In particular, if $B$ is lightface $\Sigma^0_2$ with
$\dimH B=0$ and $\xi$ is a computable irrational, then $A$ can
also be chosen lightface $\Sigma^0_2$ of dimension zero.
This includes the case $B=\{0\}$.
The passage to the divisible hull also preserves these effective
bounds: the scaled sets $\frac{1}{n!}A$ are uniformly
$\Sigma^0_2(z)$, so their union is $\Sigma^0_2(z)$.
Hence the extension may also be chosen to be a $\Q$-vector space.
\end{remark}

\section{A Borel group for all irrational multipliers}
\label{sec:borel-universal}

The second construction uses pairs of binary prefixes. One prefix
specifies the multiplier, and the other specifies the real to be
represented. We first choose a common zero-dimensional target for
all nonzero rational combinations of the generators.

\begin{lemma}\label{lem:zero-dimensional-target}
There is a dense lightface $\Pi^0_2$ set
$G\subseteq\R\setminus\{0\}$ with $\dimH G=0$.
\end{lemma}

\begin{proof}
Fix a computable enumeration $(q_\ell)_{\ell\geq1}$ of
$\Q\setminus\{0\}$. For $n,\ell\geq1$, put
\[
 r_{n,\ell}=\min\left\{\frac{|q_\ell|}{2},
                              2^{-n(n+\ell)-2}\right\},\qquad
 I_{n,\ell}=(q_\ell-r_{n,\ell},q_\ell+r_{n,\ell}).
\]
These intervals have uniformly computable rational endpoints,
contain $q_\ell$ but not $0$, and satisfy
\[
 |I_{n,\ell}|<2^{-n},\qquad
 |I_{n,\ell}|^{1/n}<2^{-n-\ell}.
\]
Put $O_n=\bigcup_{\ell\geq1}I_{n,\ell}$ and
$G=\bigcap_{n\geq1}O_n$. Each $O_n$ is open and dense, and $G$
contains $\Q\setminus\{0\}$, so $G$ is dense as well.
The sets $O_n$ are uniformly effectively open, so $G$ is
lightface $\Pi^0_2$.
For every $s>0$ and all sufficiently large $n$,
\[
 \HH^s_{2^{-n}}(G)
 \leq\sum_\ell|I_{n,\ell}|^s
 \leq\sum_\ell|I_{n,\ell}|^{1/n}<2^{-n}.
\]
Thus $\dimH G=0$ and $0\notin G$.
\end{proof}

\medskip
\noindent\textit{Idea of the proof and comparison with
Theorem~\ref{thm:main}.}
Theorem~\ref{thm:main}, through Theorem~\ref{thm:extension}, treats
one fixed irrational multiplier $\xi$. For each $\beta\in2^\omega$ it
chooses $g_\beta$ and defines $h_\beta=\beta-\xi g_\beta$, so that surjectivity is
automatic once the generators belong to $A$. We now use the same
idea simultaneously for all irrational multipliers. It suffices to
consider $\xi\in(1,2)\setminus\Q$, since the group we construct will
be a $\Q$-vector space and $A+(\xi+m)A=A+\xi A$ for $m\in\Z$.
Thus we choose a generator $g_{\xi,\beta}$ for every pair $(\xi,\beta)$ and put
\[
 h_{\xi,\beta}=\beta-\xi g_{\xi,\beta}.
\]
For each fixed $\xi$, the identity $h_{\xi,\beta}+\xi g_{\xi,\beta}=\beta$ will give
$[0,1]\subseteq A+\xi A$, and hence $A+\xi A=\R$.

The common mechanism in the two proofs is cancellation. A rational
combination of the two generators with the same parameter is
\[
 a g_{\xi,\beta}+b h_{\xi,\beta}
     =(a-b\xi)g_{\xi,\beta}+b\beta.
\]
At a finite stage, use an arbitrary test parameter $(\xi,\alpha)$
in a small parameter rectangle and the cancellation value
\[
 \widetilde u=\frac{v-b\alpha}{a-b\xi},
\]
where $v$ is a constant chosen near the value of the expression at
the old midpoint. Then $(a-b\xi)\widetilde u+b\alpha=v$,
independently of $(\xi,\alpha)$ in that rectangle. As in the first
proof, we do not retain these exact cancellation values: we use
them to choose smaller assigned dyadic intervals for the generators.

There is one new difficulty. For fixed $\xi$, the denominator
$a-b\xi$ is a nonzero constant whenever $(a,b)\neq(0,0)$.
When $\xi$ varies, $a-b\xi$ can be arbitrarily close to zero. We
therefore divide only on rectangles whose entire multiplier
interval stays away from its zero. A coefficient choice that does
not yet have this property is left untreated at that stage.
This loses no eventual requirement: for a fixed irrational $\xi$
and a fixed nonzero rational pair $(a,b)$, we have $a-b\xi\neq0$,
so sufficiently short multiplier prefixes give the required bound.
Meanwhile, refining both parameter prefixes separates every fixed
finite collection of distinct pairs $(\xi,\beta)$.

There is also a change in how we keep the sums small. In the first
proof, cancellation reduces one coefficient choice to a finite set
$S$, and we cover $B_j+S$ by intervals of small Hausdorff content.
Here we have no prescribed subgroup $B$ to retain. Instead, we
fix the zero-dimensional set $G$ from the lemma and write
$G=\bigcap_j G_j$, where the $G_j$ are decreasing dense open sets
omitting $0$. Cancellation again produces a finite set $S$.
The density of $G_j$ lets us move all of $S$ into $G_j$ by one
arbitrarily small translation $\delta$: choose
\[
 \delta\in\bigcap_{s\in S}(G_j-s)
\]
near $0$. We implement this by adding $\delta$ to the constants
on all refined pieces of one used old cell. Each sum uses that
cell exactly once, so every sum increases by exactly $\delta$.
Openness then allows small errors, and we shrink the assigned
intervals until every permitted sum lies in $G_j$.

Later steps only shrink those intervals, preserving all earlier
requirements. Every fixed nontrivial rational combination is
treated at all sufficiently late stages, so its limit belongs to
$G$. This gives both dimension zero and rational independence:
the combination cannot vanish because $0\notin G$.
For Theorem~\ref{thm:main}, the parameter space $2^\omega$ is
compact, and compact images give the $F_\sigma$ conclusion.
Here the irrational-multiplier parameter space is Polish but not
compact. Rational independence makes the maps coding finite sums
injective, so Lusin--Souslin gives Borelness instead. The sharper
complexity calculation is given in
Remark~\ref{rem:borel-complexity} below.

\medskip
\begin{proof}[Proof of Theorem~\ref{thm:borel-universal}]
Fix the sets $O_n$ above, and put
\[
 G_j=\bigcap_{n=1}^jO_n,\qquad G=\bigcap_{j\geq1}G_j.
\]
Each $G_j$ is dense and open and omits $0$. We reserve $G_j$
for these target sets in $\R$; the letters $U,V,W,T$ below denote
cells in the parameter space.
Let $\mathcal J=(1,2)\setminus\Q$ and $P=\mathcal J\times2^\omega$.
For $(\xi,\beta)\in P$ we construct a real $g_{\xi,\beta}$ and put
\[
 h_{\xi,\beta}=\beta-\xi g_{\xi,\beta}.
\]
We shall arrange that every nontrivial finite rational combination
of this indexed family belongs to $G$.

For the finite-stage requirements, use arbitrary test parameters
$(\xi,\alpha)\in P$. Since $\xi-1$ is irrational, its binary expansion
is unique.
At level $n$ the pairs of prefixes $\sigma,\tau\in2^n$ give the
partition of $P$ into $4^n$ sets
\begin{equation}\label{eq:binary-rectangles}
 U_{\sigma,\tau}
 =\{(\xi,\alpha)\in P:\sigma\preceq \xi-1,\ \tau\preceq \alpha\}.
\end{equation}
These sets are nonempty and clopen in $P$. Their numerical
coordinates lie in the closed rectangles
\[
 (1+D_\sigma)\times D_\tau.
\]
Recall that $D_\sigma=[\sigma,\sigma+2^{-|\sigma|}]$. The rational endpoints in the first coordinate are not in
$\mathcal J$. Increasing $n$ refines the partition and eventually
separates any finitely many distinct parameter pairs.

As before, choose increasing lengths $n_j,m_j$ and strings
\[
 \theta^j_{\sigma,\tau}\in2^{m_j}
 \qquad(\sigma,\tau\in2^{n_j}),
\]
each extending its predecessor
$\theta^{j-1}_{\sigma\res n_{j-1},\tau\res n_{j-1}}$.
Begin with $n_0=0$, $m_0=2$, and
$\theta^0_{\varnothing,\varnothing}=01$.
Their branchwise unions will define the final generators
$g_{\xi,\beta}$.

\medskip
\noindent\textit{Stage $j$.}
Let
\[
 Q_j=\left\{\frac{p}{q}:p\in\Z,\ q\in\N,\ |p|\leq j,
                         \ 1\leq q\leq j\right\}.
\]
These are finite increasing sets whose union is $\Q$.
Here $p,q$ are integer indices; we reserve $a_i,b_i$ for the
rational coefficients used in this proof.

List the cells inherited from
stage $j-1$, namely all $U_{\sigma,\tau}$ with
$\sigma,\tau\in2^{n_{j-1}}$, as $U_1,\ldots,U_k$,
where $k=4^{n_{j-1}}$.  If $U_i=U_{\sigma_i,\tau_i}$, put
$J_i=1+D_{\sigma_i}$.
Consider all choices $(a_i,b_i)\in Q_j^2$, not all zero, for which
\begin{equation}\label{eq:denominator}
 \kappa_i:=\min_{\xi\in J_i}|a_i-b_i\xi|>0
 \quad\text{whenever }(a_i,b_i)\neq(0,0).
\end{equation}
Each admissible coefficient list gives one substage of stage $j$.
There are at most $|Q_j|^{2k}-1$ such substages; lists failing
\eqref{eq:denominator} are omitted. We require, for each admissible list,
\begin{equation}\label{eq:universal-stage}
 \sum_{i=1}^k\bigl((a_i-b_i\xi_i)u_i+b_i\alpha_i\bigr)\in G_j
\end{equation}
whenever $(\xi_i,\alpha_i)\in U_i$ and
\[
 u_i\in
 D_{\theta^j_{(\xi_i-1)\res n_j,\,\alpha_i\res n_j}}.
\]
These are universal requirements over all test parameters
$(\xi_i,\alpha_i)\in U_i$ and all allowed $u_i$; the induction does
not select particular infinite branches. Zero pairs leave cells
unused. Some choices may fail \eqref{eq:denominator} at early
stages, but every fixed finite combination will be treated at all
sufficiently late stages.

Keep the old cells $U_i$ fixed and process these finitely many
substages in any order. Consider one such substage. Suppose the current parameter level
is $L$ and each current cell $V$ has an assigned string
$\eta_V\in2^M$. Write
\[
 t_V=\eta_V+2^{-M-1},\qquad \rho=2^{-M-1}.
\]
Refine to a common level $N>L$ so that, on every refined cell
$W\subseteq V\subseteq U_i$ with nonzero pair, the following holds.
If $W=U_{\sigma,\tau}$, put
\[
 r_W=1+\sigma,\qquad s_W=\tau,\qquad
 v_W=(a_i-b_ir_W)t_V+b_is_W.
\]
We choose $N$ large enough that, for every refined cell
$W=U_{\sigma,\tau}\subseteq V\subseteq U_i$ with
$(a_i,b_i)\neq(0,0)$ and every
\[
 (\xi,\alpha)\in(1+D_\sigma)\times D_\tau
 =[r_W,r_W+2^{-N}]\times[s_W,s_W+2^{-N}],
\]
the following estimate holds:\begin{align*}
 |(a_i-b_i\xi)t_V+b_i\alpha-v_W|
 &\leq |b_i|\bigl(|t_V|\,|\xi-r_W|+|\alpha-s_W|\bigr)\\
 &\leq |b_i|(|t_V|+1)2^{-N}<\kappa_i\rho/4.
\end{align*}
Only finitely many constants occur, so one $N$ suffices.
Consequently
\begin{equation}\label{eq:universal-cancellation}
 \left|\frac{v_W-b_i\alpha}{a_i-b_i\xi}-t_V\right|<\rho/4
\end{equation}
throughout that rectangle. Let
\[
 \widetilde u=\frac{v_W-b_i\alpha}{a_i-b_i\xi}.
\]
Then
\[
 (a_i-b_i\xi)\widetilde u+b_i\alpha
 =(a_i-b_i\xi)\frac{v_W-b_i\alpha}{a_i-b_i\xi}
   +b_i\alpha
 =v_W.
\]
For a zero pair put $v_W=0$ and make no division.

With the current coefficient choice and refinement fixed, define
\[
 S=\left\{\sum_{i=1}^kv_{W_i}\;\middle|\;
 \begin{array}{l}
 W_i\in\{U_{\sigma,\tau}:\sigma,\tau\in2^N\},\\
 W_i\subseteq U_i\quad\text{for every }1\leq i\leq k
 \end{array}\right\}.
\]
Here $(W_1,\ldots,W_k)$ ranges over all independent choices of
one level-$N$ subcell of each old cell $U_i$; $v_{W_i}$ is the
constant already attached to that subcell. Thus $S$ contains all
such sums, not just one sum arising from a selected tuple of
parameters. There are finitely many subcells, so $S$ is nonempty
and finite. Choose an index $i_*$ with nonzero pair.
Since $G_j$ is dense and open, so is
$\bigcap_{s\in S}(G_j-s)$. Choose a rational $\delta$ in this
intersection with
\[
 |\delta|<\kappa_{i_*}\rho/4.
\]
Then $S+\delta\subseteq G_j$.
Replace $v_W$ by $v'_W=v_W+\delta$ for every $W\subseteq U_{i_*}$,
and put $v'_W=v_W$ on the other cells.
Every sum uses exactly one parameter from $U_{i_*}$, so every sum
of the new constants is increased by exactly $\delta$.
For every $W\subseteq V\subseteq U_i$ with nonzero coefficient
pair, and every numerical point $(\xi,\alpha)$ in the closed
rectangle of $W$, the translated cancellation value satisfies
\begin{equation}\label{eq:translated-cancellation-bound}
 \begin{aligned}
 \left|\frac{v'_W-b_i\alpha}{a_i-b_i\xi}-t_V\right|
 &\leq
 \left|\frac{v_W-b_i\alpha}{a_i-b_i\xi}-t_V\right|
 +\frac{|v'_W-v_W|}{|a_i-b_i\xi|}\\
 &<\frac{\rho}{4}+\frac{|v'_W-v_W|}{\kappa_i}
 <\frac{\rho}{2}.
 \end{aligned}
\end{equation}
Indeed, the first term is bounded by
\eqref{eq:universal-cancellation}, while $v'_W-v_W=\delta$ when
$i=i_*$ and $v'_W-v_W=0$ otherwise. In the former case,
$|\delta|/\kappa_{i_*}<\rho/4$. Thus every translated cancellation
value remains in the interior of the old assigned interval
$D_{\eta_V}=[t_V-\rho,t_V+\rho]$.

Put
\[
 d_0=\operatorname{dist}(S+\delta,\R\setminus G_j)>0,\qquad
 C=\sum_{i=1}^k\max_{\xi\in J_i}|a_i-b_i\xi|>0,
\]
and choose $0<\varepsilon<\min\{\rho/4,d_0/(2C)\}$.
Refine to a common binary level $N'>N$ so that, for every new
cell $T=U_{\sigma,\tau}\subseteq W\subseteq V\subseteq U_i$,
where $\sigma,\tau\in2^{N'}$ and $(a_i,b_i)\neq(0,0)$,
\[
 \sup_{\substack{
   (\xi,\alpha),    (\xi',\alpha')\in(1+D_\sigma)\times D_\tau }}
 \left|
 \frac{v'_W-b_i\alpha}{a_i-b_i\xi}
 -
 \frac{v'_W-b_i\alpha'}{a_i-b_i\xi'}
 \right|
 <\frac{\varepsilon}{2}.
\]This is a finite refinement: the denominators are bounded away
from zero by \eqref{eq:denominator}, and the finitely many
expressions are uniformly continuous on their closed rectangles.

For each new cell $T=U_{\sigma,\tau}\subseteq W\subseteq V\subseteq U_i$,
where $\sigma,\tau\in2^{N'}$, the lower left corner of its closed
numerical rectangle is $(1+\sigma,\tau)$. On a cell with nonzero
coefficient pair, choose a binary string $\eta'_T$ such that
\begin{equation}\label{eq:refined-assigned-interval}
 \eta_V\preceq\eta'_T,\qquad
 \sup_{u\in D_{\eta'_T}}
 \left|u-\frac{v'_W-b_i\tau}{a_i-b_i(1+\sigma)}\right|
 <\frac{\varepsilon}{2}.
\end{equation}
The displayed center is within $\rho/2$ of $t_V$ by
\eqref{eq:translated-cancellation-bound}, so it belongs to the
interior of $D_{\eta_V}$. Take a binary expansion of this center
extending $\eta_V$ and let $\eta'_T$ be its prefix of a common
length $M'>M$ with $2^{-M'}<\varepsilon/2$.
The center belongs to $D_{\eta'_T}$, whose diameter is
$2^{-M'}$, so \eqref{eq:refined-assigned-interval} follows.
On zero-pair cells, append zeros to $\eta_V$ to reach the same
length $M'$. In every case,
\[
 \eta'_T\in2^{M'},\qquad
 D_{\eta'_T}\subseteq D_{\eta_V}.
\]
For $(\xi,\alpha)\in T$ and $u\in D_{\eta'_T}$ on a nonzero cell,
the choice of $N'$ and \eqref{eq:refined-assigned-interval} give
\[
 \begin{aligned}
 \left|u-\frac{v'_W-b_i\alpha}{a_i-b_i\xi}\right|
 &\leq
 \left|u-\frac{v'_W-b_i\tau}{a_i-b_i(1+\sigma)}\right|+
 \left|\frac{v'_W-b_i\tau}{a_i-b_i(1+\sigma)}
       -\frac{v'_W-b_i\alpha}{a_i-b_i\xi}\right|\\
 &<\frac{\varepsilon}{2}+\frac{\varepsilon}{2}
   =\varepsilon.
 \end{aligned}
\]
Consequently,
\[
 |(a_i-b_i\xi)u+b_i\alpha-v'_W|
 <|a_i-b_i\xi|\varepsilon.
\]
The contribution of a zero pair is zero.
Each resulting sum is therefore within $C\varepsilon<d_0/2$ of
$S+\delta$, and hence lies in $G_j$.
This proves \eqref{eq:universal-stage} for the present choice.
Extension of the assigned strings preserves all earlier choices.

After all admissible coefficient lists (at most $|Q_j|^{2k}-1$)
have been processed, increase the parameter
and assigned-string lengths further if necessary to make them
strictly larger than $n_{j-1}$ and $m_{j-1}$. Call them $n_j,m_j$
and call the assigned strings $\theta^j_{\sigma,\tau}$.
This also describes what to do if the list is empty.
The induction is complete.

\medskip
\noindent\textit{The generators and their rational combinations.}
After completing all stages, let $(\xi,\beta)$ range over all of
$P$ and define the final generators by
\begin{equation}\label{eq:universal-generators}
 g_{\xi,\beta}=\bigcup_{\ell\geq0}
      \theta^\ell_{(\xi-1)\res n_\ell,\,\beta\res n_\ell},
 \qquad h_{\xi,\beta}=\beta-\xi g_{\xi,\beta}.
\end{equation}
The index $\ell$ ranges over the whole construction, while
$\beta$ is an arbitrary full branch, not a stage-dependent choice.
Each union defines a real in $[1/4,1/2]$, and
\[
 h_{\xi,\beta}+\xi g_{\xi,\beta}=\beta.
\]
If two parameters are in the same level-$n_j$ cell, their
$g$-values differ by at most $2^{-m_j}$. Since those cells are
clopen, the family $g_{\xi,\beta}$ is continuous on $P$.
The family $h_{\xi,\beta}$ is continuous as well. These observations
use only the nested strings; no functions were chosen during
the induction.

Fix distinct parameters $(\xi_1,\beta_1),\ldots,(\xi_r,\beta_r)\in P$ and
nonzero rational pairs $(a_i,b_i)$. Consider
\[
 z=\sum_{i=1}^r(a_i g_{\xi_i,\beta_i}+b_i h_{\xi_i,\beta_i}).
\]
For every sufficiently large $j$, all coefficients belong to
$Q_j$ and the parameters lie in distinct old cells.
Moreover $a_i-b_i\xi_i\neq0$, since $\xi_i$ is irrational.
There is therefore a neighborhood of $\xi_i$ on which this
expression is bounded away from zero. The closed dyadic interval
of the old cell containing $\xi_i$ has length $2^{-n_{j-1}}$ and
contains $\xi_i$, so it is eventually contained in that neighborhood.
Thus \eqref{eq:denominator} holds on each used cell.
Assign its pair to that cell and zero pairs elsewhere.
This is one of the choices treated at every sufficiently large
stage. Apply \eqref{eq:universal-stage} with the test branch
equal to the corresponding $\beta_i$ and the allowed value equal
to $g_{\xi_i,\beta_i}$ on each used cell. These limit values lie in
their assigned intervals, so $z\in G_j$ at every such stage.
It follows that $z\in G$ and, in particular, $z\neq0$.

Hence the indexed family
\[
 E=\{g_{\xi,\beta}:(\xi,\beta)\in P\}
       \cup\{h_{\xi,\beta}:(\xi,\beta)\in P\}
\]
is $\Q$-linearly independent. In this definition, $(\xi,\beta)$
ranges over all of $P$, so $E$ is the full final generator set.
Put $A=\operatorname{span}_{\Q}E$.
After grouping equal parameter labels, every nonzero element of
$A$ is a sum of the type just considered. Therefore
\[
 A\setminus\{0\}\subseteq G,\qquad \dimH A=0,
\]
and $A\neq\R$.

\medskip
\noindent\textit{Borelness.}
The space $P=\mathcal J\times2^\omega$ is Polish. By the continuity
and independence just proved, the assignment
\[
 (\xi,\beta,0)\longmapsto g_{\xi,\beta},\qquad
 (\xi,\beta,1)\longmapsto h_{\xi,\beta}
\]
is a continuous injection from $P\times\{0,1\}$ into $\R$.
The Lusin--Souslin theorem makes its image $E$ Borel; see
\cite[Exercise~2E.7]{Moschovakis}.
For $r\geq1$, the set
\[
 E^{(r)}=\{(e_1,\ldots,e_r)\in E^r:e_1<\cdots<e_r\}
\]
is Borel. For each fixed
$(q_1,\ldots,q_r)\in(\Q\setminus\{0\})^r$, the continuous map
\[
 (e_1,\ldots,e_r)\longmapsto\sum_{i=1}^rq_ie_i
\]
is injective on $E^{(r)}$: independence determines the support
and its coefficients, and increasing order determines the tuple.
Its image is Borel by Lusin--Souslin. Taking the countable union
over $r$ and the nonzero rational coefficient tuples, and adjoining
$0$, proves that $A$ is Borel. Remark~\ref{rem:borel-complexity}
gives a more precise description of its Borel complexity.

\medskip
\noindent\textit{The set $X_A$.}
For $\xi\in\mathcal J$ and $\beta\in2^\omega$,
\[
 \beta=h_{\xi,\beta}+\xi g_{\xi,\beta}\in A+\xi A.
\]
Thus $A+\xi A$ contains $[0,1]$ and is a subgroup, so it is $\R$.
For an arbitrary irrational $\xi$, choose $m\in\Z$ with
$\xi+m\in\mathcal J$ and use $A+(\xi+m)A=A+\xi A$.
If $q\in\Q$, then $A+qA=A\neq\R$, since $A$ is a proper
$\Q$-vector space. Hence $X_A=\R\setminus\Q$.
The remaining complexity assertions are proved in
Remarks~\ref{rem:borel-complexity} and
\ref{rem:effective-universal} below.
\end{proof}

\begin{remark}
In the first construction the parameter space is $2^\omega$, so
the generators form a compact set. In the second it is
$((1,2)\setminus\Q)\times2^\omega$, and Borelness follows instead
from rational independence and Lusin--Souslin. Only the generators
with irrational multiplier parameters are used to generate $A$;
we do not adjoin the generators obtained by extending to the
omitted rational multipliers. Such extensions will be used only
as auxiliary objects in Remark~\ref{rem:borel-complexity}.
Retaining branch labels at dyadic endpoints is harmless in both
constructions: different labels are simply treated as different
parameters.
\end{remark}

\begin{remark}[Borel complexity]\label{rem:borel-complexity}
\label{rem:fsigma-obstruction}
The group constructed in Theorem~\ref{thm:borel-universal} is a
difference of two $F_\sigma$ sets. In particular,
\[
 A\in\boldsymbol{\Delta}^{0}_{3}
 \setminus\bigl(\boldsymbol{\Sigma}^{0}_{2}
                 \cup\boldsymbol{\Pi}^{0}_{2}\bigr).
\]
Thus its Borel rank is exactly three: it is both $G_{\delta\sigma}$
and $F_{\sigma\delta}$, but neither $F_\sigma$ nor $G_\delta$.
The point is that the binary construction extends to a compact
parameter space, while the unwanted sums involving rational
multipliers can be removed by subtracting an $F_\sigma$ set.

For this calculation only, let $\widehat P=2^\omega\times2^\omega$.
For $p=(\gamma,\beta)\in\widehat P$, put
\[
 \xi_p=1+\gamma,\qquad
 g_p=\bigcup_{\ell\geq0}
          \theta^\ell_{\gamma\res n_\ell,\,\beta\res n_\ell},
 \qquad h_p=\beta-\xi_pg_p.
\]
The strings in the construction were assigned to every pair of
finite prefixes. Hence these formulas define continuous families
on all of the compact space $\widehat P$, even when $\xi_p$ is
rational. The original generators are precisely those with
$\xi_p\notin\Q$; the auxiliary generators do not enlarge $A$.

The stage estimates hold throughout the closed numerical
rectangles, including their rational multiplier coordinates.
Consequently, the same verification as above gives
\begin{equation}\label{eq:extended-nonvanishing}
 \sum_{i=1}^r(a_i g_{p_i}+b_i h_{p_i})\in G
 \quad\text{if}\quad
 p_i\neq p_k\ (i\neq k),\qquad
 a_i-b_i\xi_{p_i}\neq0\ (1\leq i\leq r),
\end{equation}
where $r\geq1$ and all coefficients are rational. Indeed, the
distinct binary parameter labels eventually lie in distinct old
cells, and each nonzero denominator is eventually bounded away
from zero on its old closed multiplier interval. Thus the sum
belongs to $G_j$ at every sufficiently late stage, and hence to
$G$. Irrationality was used in the original verification only to
ensure these nonvanishing conditions automatically.

Let $F$ be the set of all sums satisfying the conditions in
\eqref{eq:extended-nonvanishing}. Let $F_{\mathrm{rat}}$ consist
of those sums admitting such a representation with at least one
$\xi_{p_i}\in\Q$. Both sets are $F_\sigma$. To see this, fix
$r\geq1$ and nonzero rational coefficient pairs $(a_i,b_i)$, and put
\[
 \mathcal U=\left\{(p_1,\ldots,p_r)\in\widehat P^r:
 \begin{array}{ll}
 p_i\neq p_k & (i\neq k),\\
 a_i-b_i\xi_{p_i}\neq0 & (1\leq i\leq r)
 \end{array}\right\}.
\]
It is clear  that $\mathcal U$ is open.
Therefore $\mathcal U$ is the union of the countably many
basic cylinders contained in it; write
\[
 \mathcal U=\bigcup_{\ell\geq1}\mathcal V_\ell,
 \qquad\mathcal V_\ell\text{ compact and clopen}.
\]
Its image under the continuous sum map
\[
 (p_1,\ldots,p_r)\longmapsto
 \sum_{i=1}^r(a_i g_{p_i}+b_i h_{p_i})
\]
is consequently a countable union of compact subsets of $\R$,
and hence is $F_\sigma$.

For $F_{\mathrm{rat}}$, restrict to tuples in $\mathcal U$ with
at least one rational multiplier. This restricted domain is
\[
 \bigcup_{\substack{\ell\geq1,\ 1\leq i\leq r\\
                    t\in\Q\cap[1,2]}}
 \left(\mathcal V_\ell\cap
 \{(p_1,\ldots,p_r)\in\widehat P^r:\xi_{p_i}=t\}\right).
\]
Each set inside parentheses is compact: the equality
$\xi_{p_i}=t$ defines a closed set, since the coordinate map
$(p_1,\ldots,p_r)\mapsto\xi_{p_i}$ is continuous.
The restricted domain is therefore also a countable union
of compact sets, and its image under the same sum map is
$F_\sigma$. Finally, taking the countable union over $r$ and
all rational coefficient lists proves that both $F$ and
$F_{\mathrm{rat}}$ are $F_\sigma$. Notice also that
$F_{\mathrm{rat}}\subseteq F\subseteq G$, so neither set contains
$0$.

We claim that
\begin{equation}\label{eq:borel-difference}
 A=(F\cup\{0\})\setminus F_{\mathrm{rat}}.
\end{equation}
Every nonzero element of $A$ has a representation in $F$ using
only irrational multiplier parameters. It cannot also have an
admissible representation involving a rational multiplier.
Otherwise, subtract the two representations, group equal binary
parameter labels, and discard zero coefficient pairs. At every
remaining irrational label the denominator is nonzero. At every
rational label, the pair comes only from the second representation,
so its denominator remains nonzero as well. At least one such
rational label survives. Equation~\eqref{eq:extended-nonvanishing}
would therefore put the resulting zero sum in $G$, a contradiction.
This proves $A\cap F_{\mathrm{rat}}=\varnothing$. Conversely, a
representation of an element of $F\setminus F_{\mathrm{rat}}$ uses
only irrational multiplier parameters, so that element belongs
to $A$. This proves \eqref{eq:borel-difference}, and therefore
$A\in\boldsymbol{\Delta}^{0}_{3}$.

For the lower bound, Proposition~4.6 of Ye, Yu, and Zhao
\cite{YYZ}, applied to additive subgroups, implies that
$X_B$ is meager for every Lebesgue-null $F_\sigma$ additive subgroup
$B\leq\R$. Since $\dimH A=0$, the group $A$ is Lebesgue null.
If it were $F_\sigma$, then $X_A$ would be meager, contrary to
$X_A=\R\setminus\Q$ being comeager. It is not $G_\delta$ either:
as a nonzero $\Q$-vector space it is dense in $\R$, so if it were
$G_\delta$, it would be comeager. For every $t\in\R$ the sets
$A$ and $t+A$ would then meet, giving $t\in A-A=A$ and forcing
$A=\R$, a contradiction.
\end{remark}

\begin{remark}[Effective complexity of the universal group]
\label{rem:effective-universal}
The group in Theorem~\ref{thm:borel-universal} can be chosen
lightface $\Delta^0_3$, with no oracle. More precisely, the sets
$F$ and $F_{\mathrm{rat}}$ in
Remark~\ref{rem:borel-complexity} can both be chosen lightface
$\Sigma^0_2$, and
\[
 A=(F\cup\{0\})\setminus F_{\mathrm{rat}}\in\Delta^0_3.
\]
Here the effective pointclasses have the meaning used in
Remark~\ref{rem:effective-construction}; see also
\cite[Chapter~3]{Moschovakis}. We explain both the effectivity
of the induction and that of the compact-image description.

First, Lemma~\ref{lem:zero-dimensional-target} gives a lightface
$\Pi^0_2$ target $G$, with $G_j=\bigcap_{n\leq j}O_n$ uniformly
effectively open. All finite-stage coefficients, interval
endpoints, midpoints, and constants $v_W$ are rational. In
particular, the minimum $\kappa_i$ in \eqref{eq:denominator} is
a computable rational: one checks whether the rational affine
function $a_i-b_i\xi$ vanishes on the rational interval $J_i$,
and otherwise takes the smaller absolute endpoint value.
Thus admissibility of each coefficient list is decidable,
and the first refinement level $N$ is found by rational
comparisons in the displayed estimate preceding
\eqref{eq:universal-cancellation}.

The finite set $S$ consists of rationals. Dovetail a search for
rationals $\delta,d_*$ satisfying
\[
 |\delta|<\frac{\kappa_{i_*}\rho}{4},\qquad d_*>0,\qquad
 [s+\delta-d_*,s+\delta+d_*]\subseteq G_j
 \quad(s\in S).
\]
The last condition is computably enumerable: enumerate the rational
intervals of $G_j$ until finitely many cover each of the finitely
many closed rational intervals. Density and openness give such
a pair, so the search terminates. The resulting $d_*$ is less
than the distance $d_0$ used in the proof; we do not need to
compute $d_0$ itself. Choose a rational
$0<\varepsilon<\min\{\rho/4,d_*/(2C)\}$.
The second refinement is effective as well. For example, choose
$N'>N$ so large that, for every relevant $W\subseteq U_i$,
\[
 2^{-N'}\left(
 \frac{|b_i|}{\kappa_i}
 +\frac{|b_i|(|v'_W|+|b_i|)}{\kappa_i^2}
 \right)<\frac{\varepsilon}{2}.
\]
The expression in parentheses is a rational Lipschitz bound
for $(v'_W-b_i\alpha)/(a_i-b_i\xi)$ with respect to the maximum
of the two coordinate differences. This follows by subtracting
two fractions and using $0\leq\alpha\leq1$ and the lower bound
$\kappa_i$ for the denominator. Thus the displayed
inequality implies the required oscillation bound on every
level-$N'$ rectangle. Finally, the centers in
\eqref{eq:refined-assigned-interval} are rational, so suitable
finite extensions $\eta'_T$ can be found by rational endpoint
comparisons and then extended to a common length.
Consequently $n_j,m_j$ and all assigned strings are computable,
uniformly in $j$.

It follows that the extended maps $p\mapsto g_p,h_p$ on
$\widehat P=2^\omega\times2^\omega$ from
Remark~\ref{rem:borel-complexity} are computable. Indeed, the
assigned string at stage $j$ approximates $g_p$ within
$2^{-m_j}$ using only the two length-$n_j$ parameter prefixes;
$h_p=\beta-\xi_pg_p$ is then computable by arithmetic.
No irrational parameter is fixed as an oracle: the parameter
branches are the inputs to these maps.

To verify the effective complexity of their images, fix $r\geq1$
and rational coefficient pairs $(a_i,b_i)$, and put
\[
 \Phi(p_1,\ldots,p_r)
   =\sum_{i=1}^r(a_i g_{p_i}+b_i h_{p_i}).
\]
For $p=(\gamma,\beta)$ write
$p\res m=(\gamma\res m,\beta\res m)$. For $m\geq1$, let
\[
 \mathcal K_m=\left\{(p_1,\ldots,p_r)\in\widehat P^r:
 \begin{array}{ll}
 p_i\res m\neq p_k\res m & (i\neq k),\\
 |a_i-b_i\xi_{p_i}|\geq2^{-m} & (1\leq i\leq r)
 \end{array}\right\}.
\]
The dependence of $\mathcal K_m$ and $\Phi$ on the fixed coefficient
list is suppressed in this notation. The parameter domains $\mathcal K_m$ are
uniformly $\Pi^0_1$ subsets of the effectively compact space
$\widehat P^r$, and their union is exactly the admissible domain
in \eqref{eq:extended-nonvanishing}. They are effectively
compact in the sense that their finite basic open covers can
be enumerated, uniformly in all the indices: a cover of $\mathcal K_m$,
together with its enumerated complement, covers
$\widehat P^r$, and this can be witnessed by finitely many
basic cylinders.

The same applies to each closed slice
\[
 \mathcal K_m\cap\{(p_1,\ldots,p_r):\xi_{p_i}=t\},
 \qquad 1\leq i\leq r,\quad t\in\Q\cap[1,2].
\]
Their images under the computable map $\Phi$, and the images
$\Phi(\mathcal K_m)$, are uniformly effectively compact, hence uniformly
$\Pi^0_1$ subsets of $\R$. For completeness, to semidecide that
such an image lies in an effectively open set $U$, semidecide
that its compact domain lies in the effectively open preimage
$\Phi^{-1}(U)$. Taking $U=\R\setminus J$ for closed rational
intervals $J$ enumerates the complement of the image by the
interiors of all such disjoint intervals. This justifies the
$\Pi^0_1$ claim without any effective inversion theorem.

Taking computable unions of these images over $m,r$ and all
rational coefficient lists gives $F\in\Sigma^0_2$; additionally
enumerating $i$ and $t\in\Q\cap[1,2]$ gives
$F_{\mathrm{rat}}\in\Sigma^0_2$. Equation~\eqref{eq:borel-difference}
therefore yields the asserted lightface $\Delta^0_3$ definition
of $A$. Neither this step nor the induction uses a Turing jump.
The lower bounds in Remark~\ref{rem:borel-complexity} remain
valid: this $A$ is neither $F_\sigma$ nor $G_\delta$.
\end{remark}

\section{Applications: Cartesian products and projections}
\label{sec:applications}

We now apply the preceding constructions to Cartesian products
and their projections. The basic observation is that the linear map
\[
 P_\xi:\R^2\longrightarrow\R,\qquad P_\xi(a,b)=a+\xi b,
\]
is Lipschitz and satisfies $P_\xi(A\times A)=A+\xi A$.
Thus a surjective dilate sum forces the Cartesian square to have
dimension at least one, even when $A$ itself has dimension zero.

\begin{corollary}\label{cor:positive-product}
For every $0<d<1/2$ and every irrational $\xi$, there is an
$F_\sigma$ subgroup $A$ of $(\R,+)$ such that
\[
 \dimH A=d,\qquad A+\xi A=\R,\qquad
 \dimH(A\times A)\geq1>2\dimH A.
\]
\end{corollary}

\begin{proof}
Take $A$ from Theorem~\ref{thm:main}. The Lipschitz map
$(a,b)\mapsto a+\xi b$ maps $A\times A$ onto $\R$.
Thus $\dimH(A\times A)\geq1>2d$.
\end{proof}

\begin{corollary}\label{cor:zero-product}
For every irrational $\xi$, there is an $F_\sigma$ subgroup $A$
satisfying
\[
 \dimH A=0,\qquad \dimH(A\times A)=1,\qquad A+\xi A=\R.
\]
\end{corollary}

\begin{proof}
Use Theorem~\ref{thm:main} with $d=0$. The preceding projection
argument gives the lower bound. For the upper bound, use
$A\times A\subseteq A\times\R$ and
\[
 \dimH(E\times\R)\leq\dimH E+1\qquad(E\subseteq\R).
\]
Indeed, cover a bounded part of $E$ by intervals of lengths $r_i<1$
with $\sum_i r_i^u$ arbitrarily small, where $u>\dimH E$.
Over a bounded interval in the second coordinate, each resulting
rectangle is covered by $O(r_i^{-1})$ squares of diameter $O(r_i)$.
The total $(u+1)$-content is at most a constant times
$\sum_i r_i^u$. Take a countable union of bounded parts and let
$u\downarrow\dimH E$.
\end{proof}

\begin{remark}
Starting with $B=\{0\}$ in Theorem~\ref{thm:extension} gives a group
generated by the compact set
\[
 K=\{0\}\cup\{g_\beta:\beta\in2^\omega\}
             \cup\{h_\beta:\beta\in2^\omega\},
\]
with
\[
 \dimH\bigl(m(K-K)\bigr)=0\quad(m\geq1),\qquad
 \dimH(K\times K)=1.
\]
The first assertion follows from $m(K-K)\subseteq A$.
For the second, $K+\xi K$ contains $[0,1]$, while
$K\times K\subseteq A\times A$.
\end{remark}

\begin{remark}
The group in Theorem~\ref{thm:borel-universal} also satisfies
$\dimH(A\times A)=1$. The lower bound follows by projection
using any irrational multiplier, and the upper bound follows
from $A\times A\subseteq A\times\R$, as in
Corollary~\ref{cor:zero-product}.
\end{remark}

\medskip
\noindent\textit{Connection with Marstrand's projection theorem.}
Marstrand's theorem \cite{Marstrand} states that, for a Borel set
$F\subseteq\R^2$, its orthogonal projection onto almost every
line has dimension $\min\{1,\dimH F\}$. The scalar orthogonal
projection in the unit direction $(1,\xi)/\sqrt{1+\xi^2}$ is
$P_\xi/\sqrt{1+\xi^2}$. Multiplication by this nonzero constant
does not change dimension, and the change of parameter
$\theta=\arctan\xi$ preserves null sets locally in both directions.
Consequently, for every Borel $A\subseteq\R$,
\begin{equation}\label{eq:marstrand-dilate}
 \dimH(A+\xi A)=\min\{1,\dimH(A\times A)\}
 \quad\text{for Lebesgue-almost every }\xi\in\R.
\end{equation}
Notice that the right-hand side involves $\dimH(A\times A)$,
not $2\dimH A$. The preceding examples show that these quantities
can differ even for $F_\sigma$ additive groups.

\begin{corollary}\label{cor:almost-every-dilate}
For every $d\in[0,1)$ and every prescribed irrational $\xi_0$,
there is an $F_\sigma$ subgroup $A\subseteq\R$ such that
\[
 \dimH A=d,\qquad A+\xi_0 A=\R,\qquad
 \dimH(A+\xi A)=1
 \quad\text{for Lebesgue-almost every }\xi\in\R.
\]
\end{corollary}

\begin{proof}
Choose $A$ by Theorem~\ref{thm:main} with multiplier $\xi_0$.
Surjectivity of $P_{\xi_0}$ on $A\times A$ gives
$\dimH(A\times A)\geq1$, so \eqref{eq:marstrand-dilate} applies.
\end{proof}

In particular, dimension zero of $A$ does not prevent almost all
dilate sums from having dimension one. When $\dimH(A\times A)=1$,
Marstrand's theorem by itself gives full dimension, not equality
with $\R$. Our constructions impose the stronger surjectivity
conclusion at a prescribed irrational multiplier in
Theorem~\ref{thm:main}, and at every irrational multiplier in
Theorem~\ref{thm:borel-universal}.

For the $\Q$-vector space $A$ in the latter theorem, the complete
projection-dimension profile is
\[
 \dimH P_\xi(A\times A)
 =\dimH(A+\xi A)
 =\begin{cases}
    0,&\xi\in\Q,\\
    1,&\xi\in\R\setminus\Q.
   \end{cases}
\]
Indeed, $A+\xi A=A$ for rational $\xi$, and $A+\xi A=\R$
otherwise. Thus among the finite-slope parameters $\xi$, the
exceptional set in Marstrand's dimension conclusion is exactly
$\Q$. The vertical projection is $A$ and also has dimension zero.

\section*{Acknowledgements}

The author acknowledges substantial assistance from OpenAI's
AI assistant, Codex, in the development and preparation of this
paper. This assistance included proposing and refining proof
strategies, working through technical details, and drafting and
revising the mathematical exposition and \LaTeX{} source.
These contributions went beyond language editing and arose
through sustained interaction with the author, whose questions,
corrections, and suggestions guided the work. The author takes
full responsibility for the mathematical claims, the verification
of the proofs and references, and the final text.

\end{document}